\documentclass[10pt]{article}
\usepackage{latexsym,amsmath,amsthm,verbatim,ifthen,amssymb}
\usepackage[latin1]{inputenc}
\usepackage[all]{xy}
\usepackage{graphicx}
\usepackage{epsfig}
\newdir{ >}{{}*!/-7pt/\dir{>}}

\newtheorem{theorem}{Theorem}
\newtheorem{corollary}[theorem]{Corollary}
\newtheorem{lemma}[theorem]{Lemma}
\newtheorem{proposition}[theorem]{Proposition}

\newtheorem{definition}[theorem]{Definition}
\newtheorem{remark}[theorem]{Remark}
\newtheorem{conjecture}{Conjecture}

\def\nil{{\rm{nil}\hskip1pt}}
\def\cat{{\rm{cat}\hskip1pt}}

\def\secat{{\rm{secat}\hskip1pt}}
\def\msecat{{\rm{m}_B\rm{secat}\hskip1pt}}
\def\mrelcat{{\rm{m}_B\rm{relcat}\hskip1pt}}
\def\relcat{{\rm{relcat}\hskip1pt}}

\def\TC{{\rm{TC}\hskip1pt}}

\def\cdga{{\textbf{cdga}}}

\def\Q{{\mathbb{Q}}}

\begin{document}

\title{Relative category and monoidal topological complexity
\footnotetext{This work has been supported by the RNP-Network
Program ``Applied and Computational Algebraic Topology" of the
European Science Foundation. Other support has been provided by
FEDER through the Ministerio de Educaci\'on y Ciencia projects
MTM2009-12081, MTM2010-18089 and ``Programa Operacional Factores
de Competitividade - COMPETE'',
 and by FCT -\emph{Funda\c{c}\~{a}o para a Ci\^{e}ncia e a Tecnologia} through
 projects Est-C/MAT/UI0013/2011 and PTDC/MAT/0938317/2008.
} }\author{J.G. Carrasquel-Vera\footnote{Institut de Recherche en Math\'ematique et Physique, Universit\'e Catholique de Louvain, 2 Chemin du Cyclotron, 1348 Louvain-la-Neuve, Belgium.
E-mail: \texttt{jose.carrasquel@uclouvain.be}} , J.M. Garc\'{\i}a
Calcines\footnote{Universidad de La Laguna, Facultad de
Matem\'aticas, Departamento de Matem\'atica
 Fundamental, 38271 La Laguna, Spain. E-mail:
 \texttt{jmgarcal@ull.es}}
\,and L. Vandembroucq\footnote{Centro
 de Matem\'{a}tica, Universidade do Minho, Campus de Gualtar,
 4710-057 Braga, Portugal. E-mail: \texttt{lucile@math.uminho.pt}}
}

\date{\empty}

\maketitle

\begin{abstract}
If a map $f$ has a homotopy retraction, then Doeraene and El
Haouari conjectured that the sectional category and the relative
category of $f$ are the same. In this work we discuss this
conjecture for some lower bounds of these invariants. In
particular, when we consider the diagonal map, we obtain results
supporting Iwase-Sakai's conjecture which asserts that the
topological complexity is the monoidal topological complexity.
\end{abstract}

 \vspace{0.5cm}
 \noindent{2010 \textit{Mathematics Subject Classification} : 55M30, 55P62.}\\
 \noindent{\textit{Keywords}: Sectional category, topological complexity,
 monoidal topological complexity.}
 \vspace{0.2cm}

\section*{Introduction.}

The topological complexity, which can be defined as the sectional
category of the diagonal map, is a numerical homotopy invariant
introduced by Farber in \cite{Far} for the study of the motion
planning problem in robotics. Iwase and Sakai (\cite{IS},
\cite{IwaseSakaiConjecture}) introduced a relative version, called
monoidal topological complexity, and conjectured that both notions
are the same. First results supporting Iwase-Sakai's conjecture
were given by Dranishnikov \cite{Dranishnikov} but it still
remains as an open problem.

On the other hand, Doeraene and El Haouari introduced an
approximation of the sectional category called \emph{relative
category} and proved that the difference between these two
invariants is at most one. Then they ask in
\cite{DoeraeneElHaouariConjecture} for which cases the sectional
category and the relative category agree. In such cases it would
give important information about the map under consideration. As
they have checked, in general these two invariants do not agree.
However Doeraene and El Haouari conjectured that the equality
holds as long as the map has a homotopy retraction.

In this paper we first show that Iwase-Sakai's conjecture is
actually included in Doeraene-El Haouari's conjecture. Then we
establish  analogues versions of this conjecture for several
approximations of both the sectional and relative categories.

\section{Preliminary notions and results.}

\subsection{Sectional category and topological complexity.}\label{joins}

The \emph{sectional category} of a map $f:Y\rightarrow X,$
\secat$(f)$, is the least integer $n$ (or $\infty$) such that $X$
can be covered by $n+1$ open subsets, over each of which $f$
admits a homotopy section. When $f$ is a fibration, then we can
take local strict sections, recovering the usual notion of
sectional category, or Schwarz genus \cite{Sch}, for fibrations.
The \emph{topological complexity}, $\TC(X)$, of a space $X$ in the
sense of Farber \cite{Far} is the sectional category of the path
fibration $\pi :X^I\rightarrow X\times X$, $\alpha \mapsto (\alpha
(0),\alpha (1))$. Here $X^I$ denotes the function space of all
paths $\alpha :I\rightarrow X$, provided with the compact-open
topology. Another important particular case is $\cat(X),$ the
\emph{Lusternik-Schnirelmann category} of a (pointed) space $X.$
If $PX=\{\alpha \in X^I:\alpha (0)=*\}$ is the path space of $X$
and $ev:PX\to X,$ $\alpha \mapsto \alpha (1)$ denotes
corresponding fibration, evaluation at $1$, then one has that
$\cat(X)=\secat(ev).$

The sectional category, which is a homotopy numerical invariant,
can be characterized through the iterated join of $f:Y\rightarrow
X$. Recall that, given any pair of maps
$A\stackrel{\alpha}{\longrightarrow}C\stackrel{\beta}{\longleftarrow
}B$, the \textit{join of} $\alpha$ \textit{and} $\beta$, $A*_C
B\to C$, is obtained by taking the homotopy pushout of the
homotopy pullback of $\alpha$ and $\beta,$
$$\xymatrix@C=0.7cm@R=0.7cm{ {\bullet } \ar[rr] \ar[dd] & & {B} \ar[dl] \ar[dd]^{\beta} \\
 & {A*_C B} \ar@{.>}[dr] & \\ {A} \ar[ur] \ar[rr]_{\alpha} & & {C.} }$$
 Setting  $j^0_f=f:Y\rightarrow X$ and $*^0_X Y=Y$ we can define inductively
$j^n_f:*^n_X Y\rightarrow X$ as the join of $j^{n-1}_f:*^{n-1}_X
Y\rightarrow X$ and $f:Y\rightarrow X.$ We point out that here we
are using $*^n$ to denote the join of $n+1$ copies of the
considered object. The characterization of sectional category is
then given by the following classical result, see for instance
\cite{J} or \cite{Sch}.

\begin{theorem}{\rm \label{char-sect}
Let $f:Y\rightarrow X$ be a map. If $X$ is a paracompact space,
then one has \secat$(f)\leq n$ if and only if $j^n_f:*^n_X
Y\rightarrow X$ admits a homotopy section.}
\end{theorem}

This theorem was first proved by Schwarz for a fibration $p:E\to
B,$ in which case the join $j^n_p$ can be constructed to be a
fibration and we may require the fibration $j^n_p$ to have a
strict section instead of a homotopy section. Indeed, if
$\alpha:A\to C$ and $\beta:B\to C$ are fibrations, the join map
$\alpha*_C\beta:A*_CB\to C$ can be explicitely described as
follows $A\ast_C B=A\amalg (A\times_C B\times [0,1]) \amalg B/\sim
\to C$, $\langle a,b,t\rangle  \mapsto  \alpha(a)=\beta(b)$ where
$\sim$ is given by $(a,b,t)\sim a$ if $t=0$ and $(a,b,t)\sim b$ if
$t=1$. This map is a fibration whose fibre is the ordinary join of
the fibres.

\begin{remark}{\rm
In order to avoid the unnecessary technical requirement on the
space $X$ of being paracompact, we will consider the statement in
Theorem \ref{char-sect} as the definition of sectional category of
a map. In particular, if $X$ is any topological space, then taking
$f=\pi:X^ I\to X\times X$ we have $$\TC(X)\leq n \quad
\Leftrightarrow \quad j^n_{\pi}:*^n_{X\times X} X^ I\rightarrow
X\times X \mbox{ admits a (homotopy) section.}$$}
\end{remark}

\subsection{Relative category.}

By construction of the iterated join $j^n_f: *^ n_XY\to X$ we
obtain, for each $n\geq 0$, a homotopy commutative diagram:
\begin{equation}\label{diag-relcat}
\xymatrix{Y\ar[r]^{\iota_n} \ar[rd]_f &  \ast^n_XY \ar[d]^{j^n_f}\\
&X. }
\end{equation}
In particular, for $n=0$, the map $\iota_0:Y\rightarrow Y$ is just
the identity.

\begin{definition}{\rm \cite{DoeraeneElHaouari}\label{defRelcat}
Let $f:Y\rightarrow X$ be a map. The \textit{relative category} of
$f$, denoted by $\relcat(f)$, is the least integer $n$ such that
$j^n_f$ admits a homotopy section $\sigma$ which satisfies $\sigma
f \simeq \iota_n$. }
\end{definition}

It is clear that $\secat(f)\leq \relcat(f)$. Doeraene and El
Haouari proved in \cite{DoeraeneElHaouari} that the difference
between the two invariants is at most 1:
\begin{theorem}{\rm
For any map $f:Y\rightarrow X$ one has $\secat(f)\leq
\relcat(f)\leq \secat(f)+1.$ }
\end{theorem}
\noindent They also set in \cite{DoeraeneElHaouariConjecture} the
following conjecture, that we will refer to as the \emph{D-EH
conjecture}:

\begin{conjecture}{\rm (D-EH Conjecture)
Let $f:Y\rightarrow X$ be any map. If $f:Y\rightarrow X$ admits a
homotopy retraction, then $\secat(f)= \relcat(f)$.}
\end{conjecture}

\begin{remark}{\rm
The hypothesis of the existence of a homotopy retraction cannot be
relaxed since, as pointed out in \cite{DoeraeneElHaouari}, for the
Hopf map $f:S^ 3\to S^ 2$, we have $\secat(f)=1$ while
$\relcat(f)=2$. Another example is the inclusion
$f:S^1\hookrightarrow D^2,$ for which $\secat(f)=0$ and
$\relcat(f)=1$ (observe that, in general, $\relcat(f)=0$ if and
only $f$ is a homotopy equivalence).}
\end{remark}

\subsection{Monoidal topological complexity.}
If $X$ is a topological space, we denote by $\Delta:X\to X \times
X,$ $x\mapsto (x,x)$, the diagonal map and by $s_0:X\to X^I$ the
homotopy equivalence that associates with $x\in X$ the constant
path $\hat{x}$ in $x$, then we obviously have $\pi s_0=\Delta .$

An important variant of the topological complexity is the
\emph{monoidal topological complexity}, which was introduced by
Iwase and Sakai in \cite{IS}.

\begin{definition}{\rm \cite{IS}
The \textit{monoidal topological complexity} of $X$, $\TC^M(X)$,
is the least integer $n$ such that $X\times X$ can be covered by
$n+1$ open sets $U_i\supseteq \Delta(X)$ over each of which there
exists a section $s_i$ of $\pi:X^I\to X\times X$ which satisfies
$s_i\Delta=s_0$. }
\end{definition}

Again it is clear that $\TC(X)\leq \TC^M(X)$ and Iwase-Sakai
proved that the difference between the two numbers is at most $1$:

\begin{theorem}{\rm \cite{IwaseSakaiConjecture}
For any locally finite simplicial complex  (or more generally, any
Euclidean Neighborhood Retract) $X$, one has $\TC(X)\leq
\TC^M(X)\leq \TC(X)+1.$}
\end{theorem}
Iwase and Sakai also conjectured in \cite{IwaseSakaiConjecture}
that the monoidal topological complexity coincides with the
classical topological complexity. Their conjecture will be
referred to as the \emph{I-S conjecture}

\begin{conjecture}{\rm (I-S Conjecture)
For any locally finite simplicial complex $X$, one has $\TC(X)=
\TC^M(X)$.}
\end{conjecture}

\noindent In \cite{Dranishnikov}, A. Dranishnikov shows that the
equality holds under certain restrictions on the space $X$:

\begin{theorem}{\rm \label{Dran-IS}\cite{Dranishnikov}
If $X$ is a space, then the equality $\TC(X)=\TC^M(X)$ holds in
the following cases: \begin{enumerate} \item[(i)] $X$ is a
$(q-1)$-connected simplicial complex and $\dim(X)\leq
q(\TC(X)+1)-2;$
\item[(ii)] $X$ is a connected Lie group.
\end{enumerate}}
\end{theorem}

\section{Monoidal topological complexity is a relative category.}

Part (i) of previous theorem is based on the following
characterization of $\TC^M$ (see Theorem \ref{Dran}). Using the
explicit description of the join for fibrations we can see that
Diagram (\ref{diag-relcat}) can be constructed in a commutative
way. That is, for each $n\geq 0$, there exists a commutative
diagram:
$$\xymatrix{X^I\ar[r]^(.4){\iota_n} \ar[rd]_{\pi} &  \ast^n_{X\times X}X^I \ar[d]^{j^n_\pi}\\
&X\times X. }$$

As the map $s_0:X\to X^I$, $x\mapsto \hat{x},$ satisfies $\pi
s_0=\Delta$, we set $s_n:=\iota_n s_0$ and we have for any $n\geq
0$, $j^n_\pi\, s_n=\Delta$.

\begin{theorem}{\rm \label{Dran}\cite{Dranishnikov}
If $X$ is paracompact, then $\TC^M(X)\leq n$ if and only if the
fibration $j^ n_{\pi}: *^ n_{X\times X}X^I\to X\times X$ admits a
strict section $\sigma $ such that $\sigma \Delta =s_n$.}
\end{theorem}

This gives a characterization of $\TC^M$ which is very similar to
the definition of $\relcat$. Indeed, with this notation,
$\relcat(\Delta )\leq n$ if and only if $j^ n_{\pi}: *^ n_{X\times
X}X^I\to X\times X$ admits a homotopy section $\sigma $ such that
$\sigma \Delta \simeq s_n$.

\begin{remark}{\rm
Again, as in the case of sectional category (and in particular for
topological complexity) we will consider the statement of Theorem
\ref{Dran} as the definition of $\TC^M(X)$ without requiring the
space $X$ to be paracompact.}
\end{remark}

We will prove that, under a non very restrictive condition on $X$,
the equality $\TC^M(X)=\relcat(\Delta)$ holds. In order to see
this, we use the following lemma, proved by Harper \cite{H}.

\begin{lemma}{\rm
Consider the diagram $X\stackrel{u}{\longrightarrow
}B\stackrel{\pi }{\longleftarrow }E,$ where $\pi $ is a fibration
with a strict section $\sigma :B\rightarrow E.$ Suppose
$\widehat{u}:X\rightarrow E$ is a lift of $u,$ that is,
$\widehat{u}$ satisfies $\pi \widehat{u}=u$. If $\sigma u $ is
homotopic to $\widehat{u}$, then $\sigma u$ is fibrewise homotopic
to $\widehat{u}$ (over $B$).}
\end{lemma}

Recall that a \emph{locally equiconnected space} is a space $X$ in
which the diagonal map $\Delta :X\rightarrow X\times X$ is a
(closed) cofibration. The class of locally equiconnected spaces is
large enough. For instance, CW-complexes and metrizable spaces fit
on such class.

\begin{theorem}{\rm If $X$ is a locally equiconnected space,
then $\TC^M(X)=\relcat(\Delta)$.}
\end{theorem}

\begin{proof}
Obviously, $\mbox{relcat}(\Delta )\leq \mbox{TC}^M(X).$ Now assume
$\mbox{relcat}(\Delta )=n$ and consider $\sigma :X\times
X\rightarrow *^n_{X\times X}X$ such that $j^ n_{\pi}\sigma =id$
and $\sigma \Delta \simeq s_n$. Therefore, by previous lemma, we
obtain $F:\sigma \Delta \simeq _{X\times X}s_n$ a fibrewise
homotopy over $X\times X$. Now, as $(X\times X,\Delta (X))$ is a
closed cofibred pair and $j^ n_{\pi}$ a fibration we can take a
lift in the diagram
$$\xymatrix{
{X\times X\times \{0\}\cup \Delta (X)\times I} \ar@{^(->}[d]
\ar[rr]^(.6)h & & {*^n_{X\times X}X^I} \ar[d]^{j^ n_{\pi}} \\
{X\times X\times I} \ar[rr]_{pr} \ar@{.>}[urr]^{\widetilde{h}} & &
{X\times X,} }$$ \noindent where $h$ is the map defined as
$h(x,y,0)=\sigma (x,y)$ and $h(x,x,t)=F(x,t).$ Then, defining
$\sigma ':=\widetilde{h}i_1$ we have that $j^ n_{\pi}\sigma '=id$
and $\sigma '\Delta =s_n.$ This means that $\mbox{TC}^M(X)\leq n.$
\end{proof}

\begin{corollary}{\rm
The D-EH conjecture contains the I-S conjecture.}
\end{corollary}

\begin{proof}
The diagonal map $\Delta :X\rightarrow X\times X$ admits the
projection $p_2:X\times X\rightarrow X$ as an obvious (homotopy)
retraction.
\end{proof}

Using this result we obtain a slight improvement of Theorem
\ref{Dran-IS} part (ii).

\begin{corollary}{\rm
Let $X$ be a connected CW H-space. Then
$$\mbox{TC}(X)=\mbox{TC}^M(X)=\mbox{cat}(X)=\mbox{cat}(X\times X/\Delta
(X)).$$}
\end{corollary}

\begin{proof}
It follows directly from Theorem 11 in
\cite{DoeraeneElHaouariConjecture}. See also \cite{L-S} and
\cite{GC-V}.
\end{proof}

\section{A stable version of D-EH conjecture.}

In this section we prove that the D-EH conjecture holds after
suspension. In order to make precise our statement we introduce
approximations of the sectional category and relative category of
a map in the same spirit as the $\sigma^ i$-category (see
\cite{sigmacat} or \cite{C-L-O-T}).

Let $i\geq 1$ be an integer and $f:Y\to X$ a map. By suspending
$i$ times Diagram (\ref{diag-relcat}) we get a homotopy comutative
diagram:
$$\xymatrix{\Sigma^ iY\ar[r]^{\Sigma^ i\iota_n} \ar[rd]_{\Sigma^ i f} &  \Sigma^ i\ast^n_XY \ar[d]^{\Sigma^ ij^n_f}\\
&\Sigma^ iX. }$$ We then define:
\begin{itemize}
\item $\sigma^i\secat(f)$ to be the least integer $n$ such that $\Sigma^ ij^n_f$ admits a homotopy
section;

\item $\sigma^i\relcat(f)$ to be the least integer $n$ such that $\Sigma^ ij^n_f$ admits a homotopy section
$\sigma$ which satisfies $\sigma \Sigma^ if \simeq \Sigma^
i\iota_n$.
\end{itemize}

In order to give the proof of next theorem we will use the
following well-known result:

\begin{lemma}{\rm \label{splitting} Let $Y\stackrel{f}{\to} X \stackrel{\lambda}{\to} C_f$ be a homotopy cofibre sequence.
If $f:Y\to X$ admits a homotopy retraction $r,$ then there exists
a map $\sigma: \Sigma C_ f\to \Sigma X$such that $\Sigma \lambda
\sigma\simeq id$ and $\sigma\Sigma \lambda+\Sigma f\Sigma r\simeq
id$.}
\end{lemma}

\begin{theorem}{\rm \label{sigma-version}
If $f:Y\to X$ admits a homotopy retraction then, for any $i\geq
1$, $\sigma^i \secat(f)=\sigma^i\relcat(f)$.}
\end{theorem}

\begin{proof}
Let $i\geq 1$. We just have to prove the inequality $\sigma^i
\secat(f)\geq\sigma^i\relcat(f)$. Suppose that
$\sigma^i\secat(f)\leq n$ and consider the following homotopy
commutative diagram:
$$\xymatrix{
& \Sigma^i \ast^n_XY \ar[d]^{\Sigma^ij^n_f} &\\
\Sigma^ i Y \ar[ru]^{\Sigma^ i\iota_n} \ar[r]_{\Sigma^ i f} &
\Sigma^ i X \ar[r]_{\Sigma^ i \lambda}& \Sigma^ i C_f }$$ By Lemma
\ref{splitting} we know that there exists a map $\sigma: \Sigma^i
C_ f\to \Sigma^i X$ such that $\Sigma^i \lambda \sigma\simeq id$
and $\sigma\Sigma^i \lambda+\Sigma^i f\Sigma^i r\simeq id$. Let
$s:\Sigma^ iX\to \Sigma^i \ast^n_XY$ be the homotopy section of
$\Sigma^ij^n_f$ given by the hypothesis $\sigma^i\secat(f)\leq n$
and set $s':=s\sigma\Sigma^i \lambda+\Sigma^i \iota_n\Sigma^i r$.
We then have:
$$\Sigma^ij^n_fs'=\Sigma^ij^n_fs\sigma\Sigma^i \lambda+\Sigma^ij^n_f\Sigma^i \iota_n\Sigma^i r=\sigma\Sigma^i \lambda+\Sigma^if\Sigma^i r=id.$$
Therefore $s'$ is a homotopy section of $\Sigma^ij^n_f$. In
addition, since $\Sigma^if$ is a co-H-map and $\lambda f\simeq *$,
we have
$$s'\Sigma^if\simeq s\sigma\Sigma^i \lambda\Sigma^if+ \Sigma^i \iota_n\Sigma^i r\Sigma^if\simeq \Sigma^i \iota_n.$$
This means that $\sigma^i\relcat(f)\leq n$.
\end{proof}

If $X$ is a topological space, then we can straightforwardly
define $$\sigma^i\TC(X):=\sigma ^i\secat(\Delta
);\hspace{15pt}\sigma^i\mbox{TC}^M(X):=\sigma
^i\mbox{relcat}(\Delta )$$

\begin{corollary}{\rm Let $X$ be a space. For $i\geq 1$ one has $\sigma^i\TC(X)=\sigma^i\mbox{TC}^M(X).$}
\end{corollary}

\section{A Berstein-Hilton weak version of the D-EH conjecture.}

Here we will consider weak versions of sectional and relative
categories in the sense of Berstein-Hilton and prove that the
corresponding D-EH conjecture for these invariants holds. Recall
that the relative category of a map $f:Y\rightarrow X$ has a
Whitehead characterization \cite{DoeraeneElHaouari}. Indeed, for
each $n$ we can consider the $n$-th fat wedge construction
$$t_n:T^n(f)\rightarrow X^{n+1}$$
\noindent inductively defined as follows. For $n=0$ we have
$T^0(f)=Y$ and $t_0=f:Y\rightarrow X$. If
$t_{n-1}:T^{n-1}(f)\rightarrow X^n$ is defined, then $t^n$ is the
join map
$$\xymatrix@C=0.5cm@R=0.6cm{ {\bullet } \ar[rr] \ar[dd] & & {X^n\times Y}
\ar[dl] \ar[dd]^{id_{X^n}\times f} \\
 & {T^n(f)} \ar@{.>}[dr]^{t_n} & \\
 {T^{n-1}(f)\times X} \ar[ur] \ar[rr]_{t_{n-1}\times id_X} & & {X^{n+1}}. }$$

We know that there exists a homotopy pullback (see \cite[Th.
25]{DoeraeneElHaouari} or \cite[Th. 8]{Weaksecat})
$$\xymatrix{
{*^n_X Y} \ar[rr]^{j^n_f} \ar[d]_{\varepsilon _n} & & {X}
\ar[d]^{\Delta _{n+1}} \\
{T^n(f)} \ar[rr]_{t_{n}} & & {X^{n+1}}   .}$$

Then we can also consider the following homotopy commutative
square, where $\tau _n$ is the composite $Y\rightarrow *^n_X
Y\stackrel{\varepsilon _n}{\longrightarrow }T^n(f):$

$$\xymatrix{
{Y} \ar[rr]^{\tau _n} \ar[d]_{f} & & {T^n(f)}
\ar[d]^{t_n} \\
{X} \ar[rr]_{\Delta _{n+1}} & & {X^{n+1}}   .}$$

\begin{proposition}{\rm \label{DH}\cite[Prop. 26]{DoeraeneElHaouari}
Let $f:Y\rightarrow X$ be an arbitrary map. Then
$\mbox{relcat}(f)\leq n$ if and only if there exists a map
$\varphi :X\rightarrow T^n(f)$ making commutative, up to
homotopies, the following diagram
$$\xymatrix{
{Y} \ar[rr]^{\tau _n} \ar[d]_{f} & & {T^n(f)}
\ar[d]^{t_n} \\
{X} \ar@{.>}[urr]^{\varphi}  \ar[rr]_{\Delta _{n+1}} & & {X^{n+1}}
.}$$}
\end{proposition}

In order to get a more manageable description of $T^n(f)$ and for
the sake of simplicity we will suppose in this section that
$f:Y\rightarrow X$ is a cofibration and we may therefore consider
the identification $Y\equiv f(Y).$ Observe that by a cofibration
we mean a \emph{closed} map having the usual homotopy extension
property.

\begin{proposition}{\rm \cite[Cor. 11]{Weaksecat} {\rm
Let $f:Y\hookrightarrow X$ be a cofibration. Then the $n$-th
sectional fat wedge $t_n:T^n(f)\hookrightarrow X^{n+1}$ is, up to
homotopy equivalence,
$$T^n(f)=\{(x_0,x_1,...,x_n)\in X^{n+1}\hspace{3pt}:x_i\in Y\hspace{3pt}\mbox{for some}\hspace{3pt}i\},$$
\noindent $t_n$ being the natural inclusion. Moreover, $t_n$ is a
cofibration.}}
\end{proposition}

In this case one can check that $\tau _n:Y\rightarrow T^n(f)$ is
given, up to homotopy equivalence, as $\tau _n(a)=(a,a,...,a).$ If
$\Delta _{n+1}:X\rightarrow X^{n+1}$ denotes the diagonal map,
then there is a strictly commutative diagram
$$\xymatrix{
{Y} \ar[rr]^{\tau _n} \ar@{^{(}->}[d]_{f} & & {T^n(f)}
\ar@{^{(}->}[d]^{t_n} \\
{X} \ar[rr]_{\Delta _{n+1}} & & {X^{n+1}}.  }$$

Now we introduce a refined version of Proposition \ref{DH}. In
order to do this we need the following well-known result whose
proof can be found for instance in \cite{War}.

\begin{lemma}\label{previo}{\rm
Suppose $f:Y\hookrightarrow X$ a cofibration and $\varphi
:X\rightarrow X$ a map such that $\varphi f=f$ and $\varphi \simeq
id_X.$ Then there exists a map $\psi :X\rightarrow X$ such that
$\psi f=f$ and $\psi \varphi \simeq id_X$ rel $Y.$}
\end{lemma}

\begin{proposition}\label{resumen}{\rm
Let $f:Y\hookrightarrow X$ be a cofibration. Then
$\mbox{relcat}(f)\leq n$ if and only if there exists a map $\phi
:X\rightarrow T^n(f)$ such that $\phi f=\tau _n$ and $t_n \phi
\simeq \Delta _{n+1}$ rel $Y.$}
\end{proposition}

\begin{proof}
Suppose that $\mbox{relcat}(f)\leq n$ and take $\varphi
:X\rightarrow T^n(f)$ such that $\varphi f\simeq \tau _n$ and
$t_n\varphi \simeq \Delta _{n+1}.$ Since $f$ is a cofibration we
can suppose without loss of generality that $\varphi f=\tau _n$
and $t_n\varphi \simeq \Delta _{n+1}.$ Take a homotopy $L:
t_n\varphi \simeq \Delta _{n+1}$ and consider the notation
$t_n\varphi =(\varphi _0,...,\varphi _n)$, $L=(L_0,...,L_n)$ with
$\varphi _i:X\rightarrow X$ and $L_i:X\times I\rightarrow X$ for
all $i\in \{0,1,...,n\}.$ Note that $\varphi _i f=f$ and that
$L_i:\varphi _i\simeq id_X.$ Therefore, by previous lemma, we can
find a map $\psi _i:X\rightarrow X$ such that $\psi _i f=f$ and a
homotopy $L'_i:\psi _i\varphi_i\simeq id_X$ rel $A.$ We set
$$\phi _i:=\psi _i \varphi _i:X\rightarrow X.$$
Taking into account that $(\phi _0(x),...,\phi _n(x))\in T^n(f)$
for all $x\in X,$ we obtain a map $\phi :X\rightarrow T^n(f)$ such
that $t_n\phi =(\phi _0,...,\phi _n).$ Obviously, $\phi f=\tau _n$
and $L'=(L'_0,...,L'_n)$ is a homotopy $L':t_n\phi \simeq \Delta
_{n+1}$ rel $Y.$
\end{proof}

If $f:Y\hookrightarrow X$ is a cofibration then, for each $n\geq
0$ we can take the cofibre sequence
$T^n(f)\stackrel{t_n}{\longrightarrow
}X^{n+1}\stackrel{q_n}{\longrightarrow }X^{n+1}/T^n(f)$ obtaining
a diagram
$$\xymatrix{
 & {T^n(f)} \ar@{^{(}->}[d]^{t_n} \\
 {X} \ar[r]_{\Delta _{n+1}} & {X^{n+1}} \ar[d]^{q_n} \\
  & {X^{n+1}/T^{n}(f).}
}$$

Recall from \cite{Weaksecat} that the \emph{weak sectional
category} of $f$, $\mbox{wsecat}(f),$ is defined as the least $n$
such that $q_n\Delta _{n+1}\simeq *.$

\begin{definition}{\rm
We define the \emph{weak relative category} of $f:Y\hookrightarrow
X,$ denoted $\mbox{wrelcat}(f),$ as the least $n$ such that
$q_n\Delta _{n+1}\simeq *$ rel $Y.$}
\end{definition}

\begin{proposition}{\rm
Let $f:Y\hookrightarrow X$ be a cofibration. Then the following
chain of inequalities holds
$$\mbox{nil}\hspace{2pt}H^*(X,Y)\leq \mbox{wcat}(X/Y)\leq \mbox{wrelcat}(f)\leq
\mbox{relcat}(f).$$}
\end{proposition}

\begin{proof}

>From the pushout
$$\xymatrix{
{Y} \ar[r] \ar@{^{(}->}[d]_{f} & {*} \ar[d] \\
{X} \ar[r]_p & {X/Y} }$$ \noindent we obtain the following
strictly commutative diagram, where the top square is a homotopy
pushout (see \cite[Prop. 12]{Weaksecat}) and the bottom square is
induced by the homotopy cofibre construction where the induced map
$w$ is a homotopy equivalence:
$$\xymatrix{
 & {T^n(f)} \ar[rr] \ar@{^{(}->}[d]_{t_n} & & {T^n(X/Y)} \ar@{^{(}->}[d]^{j_n} \\
{X} \ar[r]^{\Delta _{n+1}}  & {X^{n+1}} \ar[rr]_{p^{n+1}}
\ar[d]_{q_n} & & {(X/Y)^{n+1}}
\ar[d]^{q'_n} \\
 & {X^{n+1}/T^n(f)} \ar[rr]^{\simeq }_{w} & & {(X/Y)^{[n+1]}.}
}$$ Now, if $\mbox{wrelcat}(f)=n$ and we take a homotopy
$H:X\times I\rightarrow X^{n+1}/T^n(f)$ with $H:q_n\Delta
_{n+1}\simeq *$ rel $Y$, then we can define
$$\bar{H}:X/Y\times I\rightarrow (X/Y)^{[n+1]}$$
\noindent by $\bar{H}([x],t):=wH(x,t).$ Then $\bar{H}$ is a well
defined continuous map such that $\bar{H}:q'_n\Delta _{n+1}\simeq
*.$ This proves that $\mbox{wcat}(X/Y)\leq \mbox{wrelcat}(f).$
Therefore
$$\mbox{nil}\hspace{2pt}H^*(X,Y)=\mbox{cuplength}\hspace{2pt}(X/Y)\leq
\mbox{wcat}(X/Y)\leq \mbox{wrelcat}(f).$$

On the other hand, if $\mbox{relcat}(f)=n$, then by Proposition
\ref{resumen} there exists a map $\phi :X\rightarrow T^n(f)$ such
that $\phi f=\tau _n$ and $t_n\phi \simeq \Delta _{n+1}$ rel $Y.$
Therefore
$$q_n\Delta _{n+1}\simeq q_nt_n\phi =*\hspace{3pt}\mbox{rel}\hspace{3pt}Y$$
\noindent and $\mbox{wrelcat}(f)\leq n.$
\end{proof}

\begin{remark}{\rm

If $f^*$ denotes the induced homomorphism in cohomology, then
using Theorem 21(d) of \cite{Weaksecat} we immediately have the
following chain of inequalities
$$\mbox{nil}\hspace{2pt}\mbox{ker}(f^*)\leq \mbox{wsecat}(f)\leq \mbox{wrelcat}(f)\leq
\mbox{relcat}(f).$$

It is natural to ask whether
$\mbox{nil}\hspace{2pt}\mbox{ker}\hspace{2pt}(f^*)$ and
$\mbox{nil}\hspace{2pt}H^*(X,Y)$ are related or not. In Theorem
21(e) of \cite{Weaksecat} it was actually established that, if $f$
has a homotopy retraction, then
$\mbox{wsecat}(f)=\mbox{wcat}(X/Y)$ and
$$\mbox{nil}\hspace{2pt}\mbox{ker}\hspace{2pt}(f^*)=\mbox{cuplength}\hspace{2pt}(X/Y)=
\mbox{nil}\hspace{2pt}H^*(X,Y).$$}
\end{remark}

And finally, our last result in this section

\begin{theorem}{\rm
Let $f:Y\hookrightarrow X$ be a cofibration. Then
$\mbox{wrelcat}(f)=\mbox{wcat}(X/Y)$ holds. In particular, if $f$
admits a homotopy retraction, then
$\mbox{wrelcat}(f)=\mbox{wsecat}(f).$}
\end{theorem}

\begin{proof}

It only remains to prove that $\mbox{wrelcat}(f)\leq
\mbox{wcat}(X/Y).$ So suppose that $\mbox{wcat}(X/Y)=n$ and take a
homotopy $H:X/Y\times I\rightarrow (X/Y)^{[n+1]}$ such that
$H:\Delta _{n+1}q'_n\simeq *.$ Then the composite
$$X\times I\stackrel{p\times id}{\longrightarrow }X/Y\times I\stackrel{H}{\longrightarrow }
(X/Y)^{[n+1]}$$ \noindent clearly gives a homotopy $q'_n\Delta
_{n+1}p\simeq *$ rel $Y.$ But from the commutative diagram
$$\xymatrix{
  {X} \ar[rr]^p \ar[d]_{\Delta _{n+1}} & & {X/Y} \ar[d]^{\Delta _{n+1}} \\
   {X^{n+1}} \ar[rr]_{p^{n+1}} \ar[d]_{q_n} & & {(X/Y)^{n+1}} \ar[d]^{q'_n} \\
  {X^{n+1}/T^n(f)} \ar[rr]^{\simeq }_{w} & & {(X/Y)^{[n+1]}}
}$$ \noindent we have that $q'_n\Delta _{n+1}p=wq_n\Delta _{n+1}$
and therefore $wq_n\Delta _{n+1}\simeq *$ rel $Y.$

 Now take the commutative square of solids arrows
$$\xymatrix{
{Y} \ar@{^{(}->}[d]_{f} \ar[rr]^(.4){*} & & {X^{n+1}/T^n(f)} \ar[d]^{w}_{\simeq } \\
{X} \ar[rr]_(.4){*} \ar@{.>}[urr] & & {(X/Y)^{[n+1]}}.   }$$ As
$q_n\Delta _{n+1}:X\rightarrow {X^{n+1}/T^n(f)}$ and the constant
map $*:X\rightarrow {X^{n+1}/T^n(f)}$ are two liftings of this
square, by the \emph{Lifting Lemma} \cite[page 90]{B} we have that
$$q_n\Delta _{n+1}\simeq *\hspace{3pt}\mbox{rel}\hspace{2pt} Y,$$
but this means that $\mbox{wrelcat}(f)\leq n.$

The second part of the theorem follows from the fact that
$\mbox{wsecat}(f)=\mbox{wcat}(X/Y)$ when $f$ admits a homotopy
retraction (see \cite{Weaksecat}).
\end{proof}

Recall that for a locally equiconnected space $X$ the diagonal map
$\Delta :X\rightarrow X\times X$ is a cofibration. Therefore we
naturally set
$$\mbox{wTC}^M(X):=\mbox{wrelcat}(\Delta )$$ \noindent and we
directly have the following corollary. Observe that, by definition
in \cite{Weaksecat}, $\mbox{wTC}(X)=\mbox{wsecat}(\Delta ).$

\begin{corollary}{\rm
If $X$ is a locally equiconnected space, then
$$\mbox{wTC}(X)=\mbox{wTC}^M(X)=\mbox{wcat}(X\times X/\Delta(X)).$$}
\end{corollary}

Dranishnikov conjectured in \cite{Dranishnikov} that
$\mbox{TC}^M(X)=\mbox{cat}(X\times X/\Delta(X)).$ The second
equality of the above corollary can be seen as a positive answer
to a weak version of this conjecture.

\begin{remark}{\rm
>From Iwase-Sakai's characterization of $\mbox{TC}^M(X)$ in the
pointed fibrewise setting (see \cite{IwaseSakaiConjecture},
\cite{IS}) A. Franc and P. Pave\v{s}i\'{c} introduced in
\cite{F-P} some lower bounds for $\mbox{TC}^M(X).$ In particular
they defined stable and weak versions of $\mbox{TC}^M(X).$ It is
possible to check that these invariants coming from the pointed
fibrewise setting are upper bounds for our $\sigma
^i\mbox{TC}^M(X)$ and $\mbox{wTC}^M(X)$ respectively. However we
do not know whether they are the same.

}
\end{remark}

\section{The D-EH conjecture in rational homotopy theory.}

In this section we assume that  $f:Y\rightarrow X$ is a map
between simply-connected spaces of finite type over $\Q$ and we
consider the rationalization $f_0:Y_0\rightarrow X_0$. In this
context the D-EH conjecture reads as: if $f:Y\to X$ admits a
homotopy retraction then $\secat(f_0)=\relcat(f_0)$.

The sectional category of $f_0$ can be characterized as follows in
terms of any surjective model of $f$ in the category $\cdga$ of
commutative differential graded algebras:

\begin{proposition}{\rm \cite{carrasquel}
Let $f:Y\rightarrow X$ be a map with surjective $\cdga$ model
$\varphi:(A,d)\rightarrow (B,d)$ and let $K=Ker\ \varphi$. Then
$\secat(f_0)$ is the smallest $n$ for which there exists a $\cdga$
morphism $\tau$ such that $\tau\circ i= \mu_{n+1},$
\[\xymatrix{
A^{\otimes n+1}\ar@{^{(}->}[d]_{i}\ar[r]^{\mu_{n+1}}&A\\
(A^{\otimes n+1}\otimes\Lambda W_{n+1},D)\ar[ur]_{\tau}&\\
}\] where $i$ is a relative Sullivan model for the projection
$\pi:A^{\otimes m+1}\rightarrow A^{\otimes m+1}/K^{\otimes n+1}$.
}\end{proposition}

In order to estimate in terms of this data the relative category
of $f_0$ we consider the map $k_n:A\rightarrow (A\otimes\Lambda
W_{n+1},\overline{D})$ given by the pushout of $\mu_{m+1}$ and
$i$. It is easy to see that the existence of the map $\tau$ in
previous proposition is equivalent to the existence of a homotopy
retraction for $k_n$. In fact, $k_n$ is a model for the join map
$j^n_{f}:*^n_fY\rightarrow X$ and the morphism $l_n$ induced in
following diagram is a model for the map $\iota_n$ in Diagram
(\ref{diag-relcat}).

\begin{equation}\label{diag-rational}
\xymatrix@=1cm{
&&\frac{A^{\otimes n+1}}{K^{\otimes n+1}}\ar@/^3.0pc/[rddd]^{\overline{\prod_{n+1} \varphi}}\\
A^{\otimes n+1}\ar[d]_{\mu_{n+1}}\ar@{^{(}->}[rr]^{i}\ar[urr]^{\pi}&&(A^{\otimes n+1}\otimes\Lambda W_{n+1},D)\ar[u]^{\theta}_ {\simeq}\ar[d]\\
A\ar@/_2.0pc/[drrr]_{\varphi}\ar@{^{(}->}[rr]_{k_{n}}&&(A\otimes\Lambda W_{n+1},\overline{D})\ar@{-->}[dr]_{l_{n}}&\\
&&&B\\
}
\end{equation}

We can choose a relative model $i$ for $\pi$ such that the
quasi-isomorphism $\theta$ satisfies $\theta(W_{n+1})=0$. In this
case, we have, for any $\omega \in W_{n+1}$, $D\omega \in
K^{\otimes n+1} \oplus A^{\otimes n+1}\otimes \Lambda^ +W_{n+1}$
and $\overline{D}\omega \in K \oplus A\otimes \Lambda^ +W_{n+1}$.
Furthermore the induced morphism $l_n$ is such that
$\l_n(a)=\varphi(a)$ if $a\in A$ and $l_n(W_{n+1})=0$. These
remarks lead to

\begin{proposition}\label{relcat0}{\rm
Let $f\colon Y\rightarrow X$ be a map and $\varphi:A\rightarrow B$
a surjective \textbf{cdga} model for $f$ with $K=Ker\ \varphi$. If
there exists a \textbf{cdga} morphism $\tau\colon (A\otimes\Lambda
W_{n+1},\overline{D})\rightarrow A$ such that $\tau\circ k_n=Id_A$
and $\tau(W_{n+1})\subset K$ then $\relcat(f_0)\leq n$.
}\end{proposition}

Consider now the quotient map $p_n:(A,d) \to (A/K^{n+1},\bar d)$.
Let \[\bar{\varphi}: (A/K^{n+1},\bar d)\to (B,d)\] be the morphism
induced by $\varphi$. The following commutative diagram

\[\xymatrix@=1cm{
A^{\otimes n+1}\ar[d]_{\mu_{n+1}}\ar[r]^{\pi} & \frac{A^{\otimes n+1}}{K^{\otimes n+1}}\ar[d] \ar@/^3.0pc/[rdd]^{\overline{\prod_{n+1} \varphi}}\\
A\ar[r]\ar@/_2.0pc/[drr]_{\varphi}& A/K^{n+1} \ar[rd]^{\bar{\varphi}}\\
&&B,\\
}\] where the second vertical morphism is induced by the
multiplication, permits us to see that the morphism $l_n$ of
Diagram \ref{diag-rational} factors as
$$A\otimes \Lambda W_{n+1} \stackrel{\lambda_n}{\to} A/K^{n+1}\stackrel{\bar{\varphi}}{\to} B.$$
The morphism $\lambda_n$ satisfies $\lambda_nk_n=p_n$ and
$\lambda_n(W_{n+1})\subset Ker(\bar{\varphi})$. Observe now that,
if $K^{n+1}=0$, then $p_n=Id_A$ and $\bar{\varphi}=\varphi$.
Therefore the morphism $\tau:=\lambda_n$ satisfies the conditions
that give $\relcat(f_0)\leq n$ in Proposition \ref{relcat0}. We
hence obtain the following result where $\nil(K)$ denotes the
maximal length of a non trivial product in $K$.

\begin{corollary}{\rm
Let $f\colon Y\rightarrow X$ be a map and $\varphi:A\rightarrow B$
a surjective \textbf{cdga} model for $f$ with $K=Ker\ \varphi$.
Then $\relcat(f_0)\leq \nil(K)$. }\end{corollary}

We now specialize this discussion in the case of $f=\Delta:X\to
X\times X$. Since $\relcat(\Delta)=\TC^M(X)$ we write $\TC_0^M(X)$
instead of $\relcat(\Delta_0)$. A surjective \textbf{cdga} model
of $\Delta$ is given by the multiplication $\mu_A:A\otimes A\to A$
where $(A,d)$ is any \textbf{cdga} model of $X$. We thus obtain:

\begin{corollary}{\rm
Let $X$ be a space and let $(A,d)$ be a \textbf{cdga} model of
$X$. Then
 $$\TC_0(X)\leq \TC_0^M(X)\leq \nil(\ker \mu_A).$$
In particular, if $X$ admits a  \textbf{cdga} model $(A,d)$ such
that $\TC_0(X)= \nil(\ker \mu_A)$, then $\TC_0(X)=\TC_0^M(X)$.
}\end{corollary}

Using this result together with Theorem 1.4 and Corollary 2.2 of \cite{JMP} we can
exhibit two important classes of spaces for which the rational
version of the Iwase-Sakai conjecture is true:

\begin{corollary}{\rm
Let $X$ be a simply-connected space. If $X$ is formal or has its
rational homotopy, $\pi_*(X)\otimes \Q$, of finite dimension and
concentrated in odd degrees, then $\TC_0(X)=\TC_0^M(X)$.
}\end{corollary}

We finish this section with a weak version of the D-EH conjecture
that we can establish in the framework of rational homotopy
theory. As in the D-EH conjecture, we suppose that the map
$f:Y\rightarrow X$ admits a homotopy retraction.

Using standard techniques, we can consider a \textbf{cdga} model
$\varphi:(A,d)\to (B,d)$ of $f$ which admits a strict section,
\textit{i.e.} there exists a \textbf{cdga} morphism $s\colon B\to
A$ such that $\varphi s=Id_{B}$. Considering Diagram
(\ref{diag-rational}), the morphism $s$ makes $k_n$ a
$(B,d)$-module morphism and we define:

\begin{itemize}
\item $\msecat(\varphi)$ as the smallest $n$ such that $k_{n}$ admits a $(B,d)$-module retraction $r$;
\item $\mrelcat(\varphi)$ as the smallest $n$ such that $k_{n}$ admits a $(B,d)$-module retraction $r$ with $r(\Lambda^+ W_{n+1})\subset K$.
\end{itemize}
Our definition of $\mrelcat(\varphi)$ provides actually an
upper bound of a $(B,d)$-module version of the relative category in
the strict sense but it is not necessary to introduce an
intermediate notion since we have:
\begin{proposition}
$\mrelcat(\varphi)=\msecat(\varphi)$.
\end{proposition}

\begin{proof} We just have to prove that $\mrelcat(\varphi)\leq \msecat(\varphi)$.
Suppose there is a $(B,d)$-module morphism
$r\colon(A\otimes\Lambda W_{n+1},\overline{D})\to A$ such that
$r(a)=a$ for all $a\in A$. Define $r'\colon(A\otimes\Lambda
W_{n+1},\overline{D})\to A$ as $r'(a):=a$ and
$r'(a\omega):=r(a\omega)-s\varphi r(a\omega)$ for
$\omega\in\Lambda^+W_{n+1}$. It is obvious that $r'(\Lambda^+
W_{n+1})\subset K$ and that $r'(s(b)a\omega)=s(b)r'(a\omega)$. We
shall now see that $r'$ commutes with differentials. Write
$\overline{D}\omega= \alpha+\sum_i a_i\psi_i$, with $\alpha\in K$,
$\{a_i\}_i\subset A$ and $\{\psi_i\}_i\subset \Lambda^+W$. Since
$\alpha\in K$ we have $\varphi r(a\alpha)=0$. Therefore
\[r'(\overline{D}(a\omega))=r'((da)\omega)+(-1)^{|a|}
r'(a(\overline{D}\omega))=(r((da)\omega)-s\varphi
r((da)\omega)+\]\[(-1)^{|a|}\left(r(a\alpha)+r\left(a\sum_ia_i\psi_i\right)-s\varphi
r(a\alpha)-s\varphi
r\left(a\sum_ia_i\psi_i\right)\right)=\]\[r(\overline{D}(a\omega))-s\varphi
r(\overline{D}(a\omega))=\overline{D}(r'(a\omega)).\] This implies
that $\mrelcat(\varphi)\leq n$.
\end{proof}


\begin{thebibliography}{99}

\bibitem{B} {H.J. Baues.} \textit{Algebraic Homotopy}.
Cambridge Studies in Advanced Maths 15, Camb. Univ. Press (1989).

\bibitem{C-L-O-T} O. Cornea, G. Lupton, J. Oprea and D. Tanr\'e.
\textit{Lusternik-Schnirelmann category}. Math. Surveys and
Monographs, vol. \textbf{103}, AMS, 2003.

\bibitem{carrasquel} {J.G. Carrasquel-Vera.} Computations in rational
sectional category. Preprint.

\bibitem{DoeraeneElHaouari} {J.P. Doeraene and M. El Haouari.} Up to one
approximations of sectional category and topological complexity.
\textit{Topology Appl.}, \textbf{160}(5) (2013), 766-783.

\bibitem{DoeraeneElHaouariConjecture}{ J.P. Doeraene and M. El Haouari.}
When does secat equal relcat? \textit{Belg. Bull. Math. Soc.}
\textbf{20}(5) (2013), 769--776.

\bibitem{Dranishnikov}{ A. Dranishnikov.} Topological complexity of wedges and
covering maps. To appear in \textit{Proc. Amer. Math. Soc.}

\bibitem{Far} {M. Farber.} Topological complexity of motion
planning. \textit{Discrete Comput. Geom.} \textbf{29} (2003),
211-221.

\bibitem{F-P} {A. Franc and P. Pave\v{s}i\'{c}.} Lower bounds for
topological complexity. \emph{Topology Appl.} \textbf{160} (2013),
991-1004.

\bibitem{Weaksecat}{J. M. Garc\'{\i}a Calcines and L.
Vandembroucq}. Weak sectional category. \textit{Journal of the
London Math. Soc.} \textbf{82}(3) (2010), 621-642.

\bibitem{GC-V}{J. M. Garc\'{\i}a Calcines and L.
Vandembroucq}. Topological complexity and the homotopy cofibre of
the diagonal map. \textit{Math. Z.} \textbf{274} (2013), nº1-2
145-165.

\bibitem{H} {J. Harper}. A proof of Gray's conjecture. Algebraic
topology (Evanston, IL, 1988), Amer. Math. Soc., Contemp. Math.
\textbf{96}, 189-195, Providence RI (1989).


\bibitem{IS} {N. Iwase and M. Sakai}. Topological complexity is a fibrewise L.S
category. \textit{Topology Appl.} \textbf{157} (2010), 10-21.

\bibitem{IwaseSakaiConjecture} N. Iwase and M. Sakai. Erratum to
``Topological Complexity is a fibrewise LS-category".
\textit{Topology Appl.} \textbf{159} (2012), 2810-2813.

\bibitem{J} {I.M. James.}
On category in the sense of Lusternik-Schnirelmann.
\textit{Topology} \textbf{17} (1978), 331-348.

\bibitem{JMP} B. Jessup, A. Murillo and P.-E. Parent. Rational Topological Complexity.
\textit{Algebraic \& Geometric Topology} \textbf{12} (2012),
1789-1801.

\bibitem{L-S} {G. Lupton and J. Scherer.}
Topological complexity of $H$-spaces. \textit{Proc. Amer. Math.
Soc.} \textbf{141} (2013) nº5 1827--1838.

\bibitem{War} {G. Warner}. Topics in Topology and Homotopy
Theory. http://www.math.washington.edu/\~{}warner/

\bibitem{Sch} {A. Schwarz.} \textit{The genus of a fiber space},
A.M.S. Transl. \textbf{55}(1966), 49-140.

\bibitem{sigmacat} L.  Vandembroucq.  Suspension of Ganea fibrations and a Hopf
invariant, To\-po\-lo\-gy and its applications \textbf{105}
(2000), 187-200.


\end{thebibliography}
\end{document}